\documentclass[a4paper,11pt, reqno]{amsart}
\usepackage{amssymb,amsthm,amsmath}
\usepackage{cite}

\usepackage{amsmath,amsthm,amscd,amssymb}
\usepackage{latexsym}
\usepackage[colorlinks,citecolor=red,pagebackref,hypertexnames=false]{hyperref}

\theoremstyle{plain}
\newtheorem{theorem}{Theorem}[section]
\newtheorem{Def}[theorem]{Definition}
\newtheorem{lemma}[theorem]{Lemma}
\newtheorem{corollary}[theorem]{Corollary}
\newtheorem{proposition}[theorem]{Proposition}

\theoremstyle{definition}

\newtheorem{example}[theorem]{Example}

\theoremstyle{remark}

\newtheorem{case[theorem]}{Case}

\def \R{{\mathbb R}}
\def \N{{\mathbb N}}

\def \C{{\mathbb C}}

\def\norm#1.#2.{\lVert#1\rVert_{#2}}

\def\R{\mathbb R}

\def \S{{\mathcal S}}

\def \H{{\mathcal H}}

\title{The heat semigroup associated with the Jacobi--Cherednik operator and its applications}

\author{Anirudha Poria}
\address{Department of Applied Mathematics, School of Mathematics and Physics, Xi’an Jiaotong-Liverpool University, Suzhou 215123, China}
\email{Anirudha.Poria@xjtlu.edu.cn}

\author{Ramakrishnan Radha} 
\address{Department of Mathematics, Indian Institute of Technology Madras, Chennai 600036, India} 
\email{radharam@iitm.ac.in}



\date{\today}
\begin{document}
\maketitle

\begin{abstract} 
In this paper, we study the heat equation associated with the Jacobi--Cherednik operator on the real line. We establish some basic properties of the Jacobi--Cherednik heat kernel and heat semigroup. We also provide a solution to the Cauchy problem for the Jacobi--Cherednik heat operator and prove that the heat kernel is strictly positive. Then, we characterize the image of the space $L^2(\mathbb R, A_{\alpha, \beta})$ under the Jacobi--Cherednik heat semigroup as a reproducing kernel Hilbert space. As an application, we solve the modified Poisson equation and present the Jacobi--Cherednik--Markov processes.
\end{abstract}

\section{Introduction}  

During the last three decades, considerable attention has been devoted to studying the Dunkl, Jacobi--Dunkl, and Jacobi--Cherednik differential-difference operators due to their relevance in various fields of mathematics and physical applications. These operators play an important role in the harmonic analysis related to the theory of Heckman, Opdam and Cherednik (see \cite{hec91, opd95, opd00}). Let $\alpha \geq -\frac{1}{2}$, the Dunkl operator $D_\alpha$ acting on smooth functions $f$ on $\R$, is defined by (see \cite{dun91})
\[ D_\alpha f(x)= \frac{d}{dx} f(x)+ \frac{2 \alpha+1}{x} \left(  \frac{f(x)-f(-x)}{2} \right). \]
For $\alpha > -\frac{1}{2}$ and $\beta \in \R$, the Jacobi--Dunkl differential-difference operator $\Lambda_{\alpha, \beta}$ defined on $\R$ by (see \cite{hec91, opd95})
\[ \Lambda_{\alpha, \beta} f(x) = \frac{d}{dx} f(x)+ \Big[  (2\alpha + 1) \coth x + (2\beta + 1) \tanh x \Big] \frac{f(x)-f(-x)}{2} . \]
It is the Jacobi or hyperbolic analogue of the Dunkl operator and well-defined for $f \in C^1(\R)$. In this work, we consider the Jacobi--Cherednik differential-difference operator $T_{\alpha, \beta}$ defined on $\R$ by 
\[ T_{\alpha, \beta} f(x)=  \Lambda_{\alpha, \beta} f(x) - \rho f(-x), \]
where $\alpha, \beta$ satisfying $\alpha \geq \beta \geq -\frac{1}{2}$ and $\alpha > -\frac{1}{2}$, and $\rho= \alpha + \beta + 1$. The square $D^2_\alpha$ of the Dunkl operator investigated in \cite{ros98} and the square $\Lambda^2_{\alpha, \beta}$ of the Jacobi--Dunkl operator investigated in \cite{chou06}. Here, we attempt to study the operator $T^2_{\alpha, \beta}$, which is the square of the Jacobi--Cherednik operator.

In \cite{ros98}, the author studied the heat equation associated with Dunkl's Laplacian and proved that the corresponding Cauchy problem is governed by a positive one-parameter semigroup by using the maximum principle for the generalized Laplacian. Then, the authors in \cite{chou06} investigated the heat equation associated with the Jacobi--Dunkl operator and proved that the heat semigroup has a strictly positive kernel and a finite Green operator. In this paper, we generalize the results investigated in \cite{ros98, chou06} and study the heat equation associated with the Jacobi--Cherednik operator on the real line. We establish some basic properties of the Jacobi--Cherednik heat kernel and heat semigroup. We also provide an explicit solution to the Cauchy problem for the Jacobi--Cherednik heat operator using the heat kernel and prove that the heat kernel is strictly positive. Moreover, we characterize the image of the space $L^2(\mathbb R, A_{\alpha, \beta})$ under the Jacobi--Cherednik heat semigroup as a reproducing kernel Hilbert space. As an application, we solve the modified Poisson equation and present the Jacobi--Cherednik--Markov processes.

An important motivation to study the Jacobi--Cherednik operators arises from their relevance in the algebraic description of exactly solvable quantum many-body systems of Calogero--Moser--Sutherland type (see \cite{die00, hik96}) and they provide a useful tool in the study of special functions with root systems (see \cite{dun91, hec91}). These describe algebraically integrable systems in one dimension and have gained considerable interest in mathematical physics. Another motivation for the investigation of the Jacobi--Cherednik operator is to generalize the previous subjects which are bound with the physics. For a more detailed discussion, we refer to \cite{mon221, por23}. Considerable attention has been devoted to discovering applications of these operators to new contexts (see \cite{mej14, mon222, mon223, mon23, por21}). In this paper, we attempt to study the heat equation associated with the Jacobi--Cherednik operator. The Jacobi--Cherednik operators have implications in two main areas: quantum mechanics and signal analysis, and the heat equation is widely used in these areas (see \cite{hik96, ros98}). We hope that the study of the Jacobi--Cherednik heat kernel and heat semigroup makes a significant impact in these areas.

The main aim of this paper is to show that the heat equation associated with the Jacobi--Cherednik operator has a solution of the form
\[ P^{\alpha, \beta}_t f(x)= \int_{\R} p^{\alpha, \beta}_t(x, y) f(y) A_{\alpha, \beta} (y)dy, \]
where $p^{\alpha, \beta}_t(x, y)$ is a strictly positive kernel. We also prove that $\left(P^{\alpha, \beta}_t, \;t \geq 0\right)$ is a strongly continuous semigroup of operators on the space $L^2(\R, A_{\alpha, \beta})$ consisting of square-integrable functions on $\R$ with respect to the measure $A_{\alpha, \beta}(x) dx$. Using the Opdam--Cherednik transform and structure of the semigroup, we obtain a spectral representation of the heat kernel $p^{\alpha, \beta}_t$ of the form
\[ p_t^{\alpha, \beta}(x, y)= \int_{\mathbb{R}} e^{-\frac{t}{2}\left( \lambda^2+\rho^2\right)}  G^{\alpha, \beta}_\lambda(x)\;  G^{\alpha, \beta}_\lambda(-y) \; d\sigma_{\alpha, \beta}(\lambda), \quad \text{for  all\;} x,y \in \R, \; t>0, \]
where $G^{\alpha, \beta}_\lambda$ is the eigenfunction of the Jacobi--Cherednik operator $T_{\alpha, \beta}$ and $d\sigma_{\alpha, \beta}$ is the generalized Plancherel measure. 

In \cite{bar61}, Bargmann first studied the image of $L^2(\R^n)$ under Segal--Bargmann transform as a space of analytic functions, square integrable with some non-negative weight function. Over the years, this topic became popular and have drawn significant attention among many researchers, see, for example, \cite{hal04, tha07, rad09}, and references therein. Here, we consider the space $L^2(\R, A_{\alpha, \beta})$ and characterize the image of this space under the Jacobi--Cherednik heat semigroup $\left(P^{\alpha, \beta}_t, \;t \geq 0 \right)$ as a reproducing kernel Hilbert space.

Moreover, we obtain a family of Markov processes $\left( X_t, \;t \geq 0\right)$ from the semigroups $\left(P^{\alpha, \beta}_t, \;t \geq 0\right)$ on the real line, which we call Jacobi--Cherednik processes with transition probability densities given by $ p^{\alpha, \beta}_t(x, y) A_{\alpha, \beta} (y) $. Then, we show that the absolute value $\left( |X_t| , \;t \geq 0\right)$ of a Jacobi--Cherednik process is a diffusion process with infinitesimal generator 
\[ \frac{d^2}{dx^2} + \Big[ (2\alpha + 1) \coth x + (2\beta + 1) \tanh x \Big] \frac{d}{dx} + \rho^2 I. \]
Finally, we solve the modified Poisson equation $\frac{1}{2}\left(T^2_{\alpha, \beta} - \rho^2 \right) u=-f$ for all $f \in L^1(\R, A_{\alpha, \beta})$.

The paper is organized as follows. In Section \ref{sec2}, we present some preliminaries related to the Jacobi--Cherednik operator. In Section \ref{sec3}, we establish some basic properties of the Jacobi--Cherednik heat kernel and heat semigroup. We also provide a solution to the Cauchy problem for the Jacobi--Cherednik heat operator and prove that the heat kernel is strictly positive. In Section \ref{sec5}, we characterize the image of the space $L^2(\R, A_{\alpha, \beta})$ under the Jacobi--Cherednik heat semigroup $\left(P^{\alpha, \beta}_t, \;t \geq 0 \right)$ as a reproducing kernel Hilbert space. Finally, in Section \ref{sec4}, we introduce the Jacobi--Cherednik--Markov processes and solve the modified Poisson equation.

\section{Preliminaries}\label{sec2}

In this section, we give a brief overview of the Jacobi--Cherednik operator and related harmonic analysis. The main references for this section are \cite{mej14, opd95, opd00, sch08}. However, we will use the same notation as in \cite{por21}.

Let $T_{\alpha, \beta}$ denote the Jacobi--Cherednik differential--difference operator (also called the Dunkl--Cherednik operator)
\[T_{\alpha, \beta} f(x)=\frac{d}{dx} f(x)+ \Big[ 
(2\alpha + 1) \coth x + (2\beta + 1) \tanh x \Big] \frac{f(x)-f(-x)}{2} - \rho f(-x), \]
where $\alpha, \beta$ are two parameters satisfying $\alpha \geq \beta \geq -\frac{1}{2}$ and $\alpha > -\frac{1}{2}$, and $\rho= \alpha + \beta + 1$. 
It can also be written in the form
\[T_{\alpha, \beta} f(x)=\frac{d}{dx} f(x)+ \frac{A'_{\alpha, \beta} (x)}{A_{\alpha, \beta} (x)} \left(  \frac{f(x)-f(-x)}{2} \right) - \rho f(-x), \]
where \[ A_{\alpha, \beta} (x)= (\sinh |x| )^{2 \alpha+1} (\cosh x )^{2 \beta+1}. \] 

Let $\lambda \in \C$. The Opdam hypergeometric functions $G^{\alpha, \beta}_\lambda$ on $\R$ are eigenfunctions $T_{\alpha, \beta} G^{\alpha, \beta}_\lambda(x)=i \lambda  G^{\alpha, \beta}_\lambda(x)$ of $T_{\alpha, \beta}$ that are normalized such that $G^{\alpha, \beta}_\lambda(0)=1$. The eigenfunction $G^{\alpha, \beta}_\lambda$ is given by
\[G^{\alpha, \beta}_\lambda (x)= \varphi^{\alpha, \beta}_\lambda (x) - \frac{1}{\rho - i \lambda} \frac{d}{dx}\varphi^{\alpha, \beta}_\lambda (x)=\varphi^{\alpha, \beta}_\lambda (x)+ \frac{\rho+i \lambda}{4(\alpha+1)} \sinh 2x \; \varphi^{\alpha+1, \beta+1}_\lambda (x),  \]
where $\varphi^{\alpha, \beta}_\lambda (x)={}_2F_1 \left(\frac{\rho+i \lambda}{2}, \frac{\rho-i \lambda}{2} ; \alpha+1; -\sinh^2 x \right) $ is the hypergeometric function. For $\lambda \in \R$, the function $G^{\alpha, \beta}_\lambda$ is real and strictly positive.

Let $p$ and $q$ be polynomials of degree $m$ and $n$. Then there exists a positive constant $C$ such that for every $\lambda \in \mathbb{C}$ and $x \in \mathbb{R}$, we have
\begin{equation}\label{eq01}
\left|p\left(\frac{\partial}{\partial \lambda}\right) q\left(\frac{\partial}{\partial x}\right) G^{\alpha, \beta}_\lambda(x)\right| \leq  C(1+|x|)^n(1+|\lambda|)^m e^{-\rho|x|} e^{|\mathrm{Im} \lambda||x|} .
\end{equation}
In particular, for every $ \lambda \in \C$ and $x \in  \R$, the eigenfunction
$G^{\alpha, \beta}_\lambda$ satisfy
\begin{equation}\label{eq04}
\left| \frac{\partial^n}{\partial x^n} G^{\alpha, \beta}_\lambda(x)\right| \leq  C(1+|x|)^n e^{-\rho|x|} e^{|\mathrm{Im} \lambda||x|}.
\end{equation}
The function $G^{\alpha, \beta}_0$ is strictly positive and bounded above by $C(1+x) e^{-\rho x}$, if $x \geq 0$ and $C e^{\rho x}$, if $x \leq 0$. Also, for every $\lambda \in \mathbb{R}$ and $x \in \mathbb{R}$, we have $  | G^{\alpha, \beta}_\lambda(x) | \leq G^{\alpha, \beta}_0(x)$. The hypergeometric functions $G^{\alpha, \beta}_\lambda$ satisfy $G^{\alpha, \beta}_{\lambda }(tx)=G^{\alpha, \beta}_{\lambda  t}(x)$,  for every $ \lambda \in \C $ and $x, t \in  \R$. In particular, for $t=0$, we have $G^{\alpha, \beta}_{0}(x)= G^{\alpha, \beta}_{\lambda }(0)=1$. For a detailed study of the theory of Opdam hypergeometric functions $G^{\alpha, \beta}_\lambda$, we refer to \cite{sch08}.

Let us denote by $C_c (\R)$ the space of continuous functions on $\R$ with compact support. The Opdam--Cherednik transform is the Fourier transform in the trigonometric Dunkl setting, and it is defined as follows.
\begin{Def}
Let $\alpha \geq \beta \geq -\frac{1}{2}$ with $\alpha > -\frac{1}{2}$. The Opdam--Cherednik transform $\mathcal{H}_{\alpha, \beta} (f)$ of a function $f \in C_c(\R)$ is defined by
\[ \H_{\alpha, \beta} (f) (\lambda)=\int_{\R} f(x)\; G^{\alpha, \beta}_\lambda(-x)\; A_{\alpha, \beta} (x) dx \quad \text{for all } \lambda \in \C. \]
The inverse Opdam--Cherednik transform for a suitable function $g$ on $\R$ is given by
\[ \H_{\alpha, \beta}^{-1} (g) (x)= \int_{\R} g(\lambda)\; G^{\alpha, \beta}_\lambda(x)\; d\sigma_{\alpha, \beta}(\lambda) \quad \text{for all } x \in \R, \]
where $$d\sigma_{\alpha, \beta}(\lambda)= \left(1- \dfrac{\rho}{i \lambda} \right) \dfrac{d \lambda}{8 \pi |C_{\alpha, \beta}(\lambda)|^2}$$ and 
$$C_{\alpha, \beta}(\lambda)= \dfrac{2^{\rho - i \lambda} \Gamma(\alpha+1) \Gamma(i \lambda)}{\Gamma \left(\frac{\rho + i \lambda}{2}\right)\; \Gamma\left(\frac{\alpha - \beta+1+i \lambda}{2}\right)}, \quad \lambda \in \C \setminus i \mathbb{N}.$$
\end{Def}

The Plancherel formula is given by 
\begin{equation}\label{eq03}
\int_{\R} |f(x)|^2 A_{\alpha, \beta}(x) dx=\int_\R \H_{\alpha, \beta} (f)(\lambda) \overline{\H_{\alpha, \beta} ( \check{f})(-\lambda)} \; d \sigma_{\alpha, \beta} (\lambda),
\end{equation}
where $\check{f}(x):=f(-x)$. Since $G^{\alpha, \beta}_{\lambda }(tx)=G^{\alpha, \beta}_{\lambda  t}(x)$, for every $ \lambda \in \C $ and $x, t \in  \R$, using the definition of $\mathcal{H}_{\alpha, \beta}$, we obtain $\mathcal{H}_{\alpha, \beta}(\check{f})(-\lambda)=\mathcal{H}_{\alpha, \beta}(f)(\lambda)$. Therefore, we can rewrite the Plancherel formula (\ref{eq03}) as follows: 
\begin{equation}\label{newpf}
\int_{\mathbb{R}}|f(x)|^{2} A_{\alpha, \beta}(x) d x=\int_{\mathbb{R}} \left| \mathcal{H}_{\alpha, \beta}(f)(\lambda) \right|^2 d \sigma_{\alpha, \beta}(\lambda).
\end{equation}

Let $L^p(\R,A_{\alpha, \beta} )$ (resp. $L^p(\R, \sigma_{\alpha, \beta} )$), $p \in [1, \infty] $, denote the $L^p$-spaces corresponding to the measure $A_{\alpha, \beta}(x) dx$ (resp. $d | \sigma_{\alpha, \beta} |(x)$). The Schwartz space $\S_{\alpha, \beta}(\R )=(\cosh x )^{-\rho} \S(\R)$ is defined as the space of all differentiable functions $f$ such that 
$$ \sup_{x \in \R} \; (1+|x|)^m e^{\rho |x|} \left|\frac{d^n}{dx^n} f(x) \right|<\infty,$$ 
for all $m, n \in \N_0 = \N \cup \{0\}$, equipped with the obvious seminorms. The Opdam--Cherednik transform $\H_{\alpha, \beta}$ and its inverse $\H_{\alpha, \beta}^{-1}$ are topological isomorphisms between the space $\S_{\alpha, \beta}(\R )$ and the space $\S(\R)$ (see \cite{sch08}, Theorem 4.1).

Let $\S^r_{\alpha, \beta} (\mathbb{R})$ be the weighted Schwartz space defined by $\S^r_{\alpha, \beta}(\R )=(\cosh x )^{-\frac{\rho}{r}} \S(\R)$, equipped with the topology induced by the seminorms 
\[ \mathcal{N}_{m, n}^r(f)=\sup _{\substack{x \in \mathbb{R} \\ 0 \leq k \leq n}}(\cosh x)^{\frac{ \rho}{r}}\left(1+x^2\right)^m\left|\frac{\mathrm{d}^k}{\mathrm{~d} x^k} f(x)\right| < \infty, \quad \text{for \;all\;} m, n \in \N. \]
For $0<r<1$, let $\S\left(\Omega_{\varepsilon(r)}\right)$ be the extended Schwartz space of all functions $h$ that are analytic in the interior of $\Omega_{\varepsilon(r)}=\{z \in \mathbb{C} \; : \; |\operatorname{Im}(z)| \leq \varepsilon(r)\}$, where $\varepsilon(r)=\left(\frac{1}{r}-1\right) \rho $, and such that $h$ together with all its derivatives extend continuously to $\Omega_{\varepsilon(r)}$ and satisfy 
\[\sup _{\lambda \in \Omega_{\varepsilon(r)}}(1+|\lambda|)^m \left|\frac{\mathrm{d}^n}{\mathrm{~d} \lambda^n} h(\lambda)\right|<\infty, \quad \text{for \;all\;} m, n \in \N. \]
If  $r=1$, the space $\S\left(\Omega_{\varepsilon(1)}\right)$ is the usual Schwartz space $\S(\R)$. For all $0<r \leq 1$, the Opdam--Cherednik transform $\H_{\alpha, \beta}$ is a topological isomorphism between
\begin{equation}\label{eq2}
\S^r_{\alpha, \beta}(\mathbb{R}) \quad \text{and} \quad
\S\left(\Omega_{\varepsilon(r)}\right).
\end{equation}
Also, the Opdam--Cherednik transform $\H_{\alpha, \beta}$ extends uniquely to an isometric isomorphism from $L^2(\R, A_{\alpha, \beta} )$ onto $L^2(\R, \sigma_{\alpha, \beta})$.

Since $T_{\alpha, \beta} G^{\alpha, \beta}_\lambda(x)=i \lambda  G^{\alpha, \beta}_\lambda(x)$, we have
\begin{equation}\label{eq02}
T^2_{\alpha, \beta} G^{\alpha, \beta}_\lambda(x)=i \lambda \; T_{\alpha, \beta} G^{\alpha, \beta}_\lambda(x)
=i \lambda \left( i \lambda G^{\alpha, \beta}_\lambda(x) \right)
=- \lambda^2  G^{\alpha, \beta}_\lambda(x).
\end{equation}

\section{The heat equation for the Jacobi--Cherednik operator}\label{sec3}

\subsection{The differential--difference operator $T^2_{\alpha, \beta}$}
Here, we study the operator $T^2_{\alpha, \beta}$. Let 
\[ Df(x)= \frac{d}{dx} f(x), \quad Mf(x) = \frac{f(x)-f(-x)}{2}, \quad \text{and} \quad Sf(x)=f(-x). \]
Then, for $f \in C^2(\R)$, using the fact that $\frac{A'_{\alpha, \beta} }{A_{\alpha, \beta} }$ is an odd function, we obtain
\begin{eqnarray}\label{eq05}
&& T^2_{\alpha, \beta} f(x)
=  T_{\alpha, \beta} \left( T_{\alpha, \beta} f(x) \right)  \nonumber \\ 
&& = \left[ D + \frac{A'_{\alpha, \beta} }{A_{\alpha, \beta} }\;  M  - \rho S \right] \left( \frac{d}{dx} f(x)+ \frac{A'_{\alpha, \beta} (x)}{A_{\alpha, \beta} (x)} \; Mf(x) - \rho \;Sf(x) \right) \nonumber \\ 
&& = \frac{d^2}{dx^2} f(x) + \frac{d}{dx} \left( \frac{A'_{\alpha, \beta}}{A_{\alpha, \beta} } \; Mf \right) (x) + \rho f'(-x) + \frac{A'_{\alpha, \beta} (x)}{A_{\alpha, \beta} (x)} Mf'(x) - \rho \frac{A'_{\alpha, \beta} (x)}{A_{\alpha, \beta} (x)} M f(-x) \nonumber \\ 
&& \quad  + \frac{A'_{\alpha, \beta} (x)}{A_{\alpha, \beta} (x)} \; M  \left( \frac{A'_{\alpha, \beta} }{A_{\alpha, \beta} } \; Mf \right) (x)  - \rho  f'(-x) + \rho \frac{A'_{\alpha, \beta} (x)}{A_{\alpha, \beta} (x)} Mf(-x) + \rho^2 f(x) \nonumber \\ 
&& = \frac{d^2}{dx^2} f(x) + \frac{d}{dx} \left( \frac{A'_{\alpha, \beta} }{A_{\alpha, \beta} } \right) (x) Mf(x) + \frac{A'_{\alpha, \beta} (x)}{A_{\alpha, \beta} (x)} f'(x) + \rho^2 f(x).
\end{eqnarray}
Now, we consider a subspace of $C^2(\R)$ consisting of functions satisfying 
\begin{equation}\label{eq9}
\frac{d^2}{dx^2} f(x_0) + \rho^2 f(x_0) \leq 0 \quad \text{if} \quad \frac{d^2}{dx^2} f(x_0) \leq 0,
\end{equation}
for some point $x_0 \in \R$. We denote this subspace as $C^2_\rho (\R)$, and show that $C^2_\rho (\R)$ is non-empty. We first give an example to show that $C^2_\rho (\R)$ is non-empty. 
\begin{example}
Let $\alpha=\frac{1}{2}$ and $\beta=-\frac{1}{2}$. Then $\rho=1$. Let us take the Gaussian function $f(x)=e^{-x^2}$ and $x_0=0$. Then $\frac{d^2}{dx^2} f(x_0)=-2 <0$, and the function $f$ satisfies the condition $\frac{d^2}{dx^2} f(x_0) + \rho^2 f(x_0)=-1 < 0$.
\end{example}

\begin{proposition}\label{pro1}
Let $\Omega \subset \mathbb{R}$ be open and symmetric. If a real valued function $f \in C^2_\rho(\Omega)$ attains an absolute maximum at $x_0 \in \Omega$,  i.e., $f(x_0)=\sup _{x \in \Omega} f(x)$, then $T^2_{\alpha, \beta} f(x_0 ) \leq 0$.
\end{proposition}
\begin{proof}
Since $f\left(x_0\right)=\sup _{x \in \Omega} f(x)$, we have $\frac{\mathrm{d}}{\mathrm{d} x} f\left(x_0\right)=0$ and $\frac{\mathrm{d}^2}{\mathrm{~d} x^2} f\left(x_0\right) \leq 0$. Also, as $f \in C^2_\rho(\Omega)$, we get $\frac{\mathrm{d}^2}{\mathrm{~d}x^2} f(x_0) + \rho^2 f(x_0) \leq 0$. Since $\frac{A'_{\alpha, \beta}}{A_{\alpha, \beta}}$ is a decreasing function, we have $\frac{\mathrm{d}}{\mathrm{d} x} \left( \frac{A'_{\alpha, \beta}}{A_{\alpha, \beta}} \right)(x_0) \leq 0$, and $ Mf(x_0)= \frac{f\left(x_0\right)-f\left(-x_0\right)}{2} \geq 0$. Hence, from (\ref{eq05}), we get $T^2_{\alpha, \beta} f(x_0 ) \leq 0$.
\end{proof}  
  
\begin{lemma}
For every $f \in \S^r_{\alpha, \beta} (\mathbb{R})$ and $\lambda \in \R$, we have 
\begin{equation}\label{eq06}
\H_{\alpha, \beta} \left( T^2_{\alpha, \beta} f \right) (\lambda )= - \lambda^2 \H_{\alpha, \beta} (f) (\lambda).
\end{equation}
\end{lemma}
\begin{proof}
Using the definition of $\H_{\alpha, \beta}$, integrating by parts and change of variables, we obtain
\begin{eqnarray*}
&& \H_{\alpha, \beta} \left( T_{\alpha, \beta} f \right) (\lambda ) \\
&& = \int_{\R} T_{\alpha, \beta}f(x)\; G^{\alpha, \beta}_\lambda(-x)\; A_{\alpha, \beta} (x) dx \\
&& = \int_{\R} \left( \frac{d}{dx} f(x)+ \frac{A'_{\alpha, \beta} (x)}{A_{\alpha, \beta} (x)} \left( \frac{f(x)-f(-x)}{2} \right) - \rho f(-x) \right) G^{\alpha, \beta}_\lambda(-x)\; A_{\alpha, \beta} (x) dx \\
&& = \int_{\R} \left( \frac{d}{dx} \left( f A_{\alpha, \beta} \right)(x) - A'_{\alpha, \beta} (x)\left(\frac{f(x)+f(-x)}{2} \right) - \rho A_{\alpha, \beta} (x) f(-x) \right) G^{\alpha, \beta}_\lambda(-x) dx \\ 
&& = - \int_{\R} f(x) \left(  \frac{d}{dx} G^{\alpha, \beta}_\lambda(-x)  + \frac{A'_{\alpha, \beta} (x)}{A_{\alpha, \beta} (x)} \left( \frac{G^{\alpha, \beta}_\lambda(-x) - G^{\alpha, \beta}_\lambda(x)}{2} \right) - \rho G^{\alpha, \beta}_\lambda(x) \right) A_{\alpha, \beta} (x) dx \\
&& =  - \int_{\R} f(x) \; T_{\alpha, \beta}G^{\alpha, \beta}_\lambda(-x)  \; A_{\alpha, \beta} (x) dx. 
\end{eqnarray*}
Now, using the relation (\ref{eq02}), we get
\begin{eqnarray*}
\H_{\alpha, \beta} \left( T^2_{\alpha, \beta} f \right) (\lambda ) 
& = & \int_{\R} T^2_{\alpha, \beta}f(x)\; G^{\alpha, \beta}_\lambda(-x)\; A_{\alpha, \beta} (x) dx \\
& = &  - \int_{\R} T_{\alpha, \beta}f(x) \; T_{\alpha, \beta}G^{\alpha, \beta}_\lambda(-x)  \; A_{\alpha, \beta} (x) dx \\
& = & \int_{\R} f(x) \; T^2_{\alpha, \beta}G^{\alpha, \beta}_\lambda(-x)  \; A_{\alpha, \beta} (x) dx \\
& = & - \lambda^2 \int_{\R} f(x) \; G^{\alpha, \beta}_\lambda(-x)  \; A_{\alpha, \beta} (x) dx \\
& = & - \lambda^2 \H_{\alpha, \beta} (f) (\lambda).
\end{eqnarray*}
\end{proof}

\subsection{The Jacobi--Cherednik heat kernel}

In this subsection, we consider the heat equation with Cauchy data for the Heckman--Opdam Laplacian. The Heckman--Opdam Laplacian (or the modified Laplacian) $\mathcal{D}_{\alpha, \beta}$ is defined by 
\[ \mathcal{D}_{\alpha, \beta}=\frac{1}{2} \left( T^2_{\alpha, \beta} - \rho^2\right).\]
The Jacobi--Cherednik heat operator $\mathrm{H}$ is defined by
\[ \mathrm{H}=\frac{\partial}{\partial t} -\mathcal{D}_{\alpha, \beta} \]
on $C^2(\Omega \times (0, T))$, where $\Omega \subset \mathbb{R}$ is open and $T>0$. Let $C_b(\mathbb{R})$ denote the space of continuous and bounded functions on $\mathbb{R}$. Here, we consider the homogeneous Cauchy problem: Given $f  \in C_b(\mathbb{R})$, find $u \in C^2(\mathbb{R} \times [0, T]) \cap C (\mathbb{R} \times [0, T])$, such that 
\begin{equation}\label{eq07}
\left\{\begin{array}{l}
\mathrm{H} u =0, \quad \text{on} \;\; \mathbb{R} \times (0, T), \\
u(\cdot, 0)=f.
\end{array}\right.
\end{equation}
Now, we define the function $\mathcal{F}_t^{\alpha, \beta}$ by
\begin{equation}\label{eq08}
\mathcal{F}_t^{\alpha, \beta}(x)= \H_{\alpha, \beta}^{-1}  \left(  e^{-\frac{t}{2}\left( \lambda^2 + \rho^2 \right)}  \right)(x), \quad x \in \R, \; t>0,
\end{equation}
which is the fundamental solution for the Jacobi--Cherednik heat equation. First, we prove the following proposition.
\begin{proposition}\label{pro2}
For any $t>0$, the function $\mathcal{F}_t^{\alpha, \beta} \in \S^r_{\alpha, \beta} (\mathbb{R})$, for every $0< r \leq 1$. Further, the function $(x, t) \mapsto \mathcal{F}_t^{\alpha, \beta} (x)$, is infinitely differentiable on $ \mathbb{R} \times (0, \infty)$ and is a solution of $\mathrm{H} u = 0$.
\end{proposition} 
\begin{proof}
Using the definition of the inverse Opdam--Cherednik transform, for every $x \in \mathbb{R}$ and $t>0$, we have 
\[ \mathcal{F}_t^{\alpha, \beta} (x) = \int_{\R} e^{-\frac{t}{2}\left( \lambda^2 + \rho^2 \right)} G^{\alpha, \beta}_\lambda(x)\; d\sigma_{\alpha, \beta}(\lambda). \] 
Since the function $\lambda \mapsto e^{-\frac{t}{2}\left( \lambda^2 + \rho^2 \right)}$ is in $\S\left(\Omega_{\varepsilon(r)}\right)$, by (\ref{eq2}), $\H_{\alpha, \beta}^{-1} \left( e^{-\frac{t}{2}\left( \lambda^2 + \rho^2 \right)} \right) = \mathcal{F}_t^{\alpha, \beta} \in \S^r_{\alpha, \beta} (\mathbb{R})$. Hence, $\mathcal{F}_t^{\alpha, \beta} \in L^1(\R, A_{\alpha, \beta})$. Now, from (\ref{eq04}), we get sufficient decay properties of the derivatives of $e^{-\frac{t}{2}\left( \lambda^2 + \rho^2 \right)} G^{\alpha, \beta}_\lambda(x)$. We get the necessary differentiations of $\mathcal{F}_t^{\alpha, \beta} (x)$ under the integral sign by using the dominated convergence theorem. Therefore, $\mathcal{F}_t^{\alpha, \beta} (x)$ is infinitely differentiable on $\mathbb{R} \times (0, \infty)$. Finally, interchanging the operator $T^2_{\alpha, \beta}$ with the integral sign and using (\ref{eq02}), we get $\mathrm{H} \mathcal{F}_t^{\alpha, \beta} (x) = 0$. 
\end{proof} 

Next, for $x, y \in \R$ and $t>0$, we define the heat kernel $p_t^{\alpha, \beta}(x, y)$ using the fundamental solution of the Jacobi--Cherednik heat equation and the generalized Jacobi--Cherednik translations. For $x \in \R$ and $f \in \S^r_{\alpha, \beta} (\mathbb{R})$, the Jacobi--Cherednik translation operator $\tau^{\alpha, \beta}_x$ is defined by
\begin{equation}\label{eq09}
\tau^{\alpha, \beta}_x f(y) = \int_{\mathbb{R}} \H_{\alpha, \beta} (f) (\lambda) \; G^{\alpha, \beta}_{-\lambda}(x) \; G^{\alpha, \beta}_{-\lambda}(y) \;  d\sigma_{\alpha, \beta}(\lambda), \quad \text{for all } y \in \R.
\end{equation}
For $x, y \in \R$ and $t>0$, the Jacobi--Cherednik heat kernel $p_t^{\alpha, \beta}(x, y)$ is given by
\begin{equation}\label{eq1}
p_t^{\alpha, \beta}(x, y) =  \tau^{\alpha, \beta}_{-x} \mathcal{F}_t^{\alpha, \beta}(y).
\end{equation}
In the following proposition, we present some basic properties of $p_t^{\alpha,\beta}$.
\begin{proposition}\label{pro3}
The Jacobi--Cherednik heat kernel $p_t^{\alpha, \beta}$ satisfies the following properties:
\begin{itemize}
\item[(a)] For all $x, y \in \R$ and $t>0$, we have
\begin{equation}\label{eq3}
p_t^{\alpha, \beta}(x, y)= \int_{\mathbb{R}} e^{-\frac{t}{2}\left( \lambda^2+\rho^2\right)}  G^{\alpha, \beta}_\lambda(x)\;  G^{\alpha, \beta}_\lambda(-y) \; d\sigma_{\alpha, \beta}(\lambda).
\end{equation}

\item[(b)] For all $x \in \R$ and $t>0$, $p_t^{\alpha, \beta}(x, 0) = \mathcal{F}_t^{\alpha, \beta}(x)$. 

\item[(c)] For all $x, y \in \R$ and $t>0$, $p_t^{\alpha, \beta}(x, y)=p_t^{\alpha, \beta}(-y, -x)$.

\item[(d)] For all $x \in \R$, $t>0$ and $\lambda \in \mathbb{C}$, we have $p_t^{\alpha, \beta}(x, \cdot) \in \S^r_{\alpha, \beta} (\mathbb{R})$ and 
\[ \H_{\alpha, \beta} \left( p_t^{\alpha, \beta}(x, - \; \cdot) \right) (\lambda) = e^{-\frac{t}{2}\left( \lambda^2+\rho^2\right)} G^{\alpha, \beta}_\lambda(x). \] 

\item[(e)] For all $x \in \R$ and $t>0$, we have
\begin{equation}\label{eq4}
\int_{\R} p_t^{\alpha, \beta}(x, y) A_{\alpha, \beta} (y) dy = e^{-\frac{t}{2} \rho^2}.
\end{equation}

\item[(f)] For fixed $y \in \R$, the function $u(x, t) = p_t^{\alpha, \beta}(x, y)$ solves the Jacobi--Cherednik heat equation $\mathrm{H} u=0$ on $\R \times (0, \infty)$.
\end{itemize}
\end{proposition}
\begin{proof}
(a) Since $G^{\alpha, \beta}_{\lambda}(\mu x)=G^{\alpha, \beta}_{\lambda \mu}(x)$, for $x, \lambda, \mu \in \R$, using (\ref{eq08}) and (\ref{eq09}), we get
\begin{eqnarray*}
p_t^{\alpha, \beta}(x, y) =  \tau^{\alpha, \beta}_{-x} \mathcal{F}_t^{\alpha, \beta}(y)
& = & \int_{\mathbb{R}} \H_{\alpha, \beta} \left(\mathcal{F}_t^{\alpha, \beta} \right) (\lambda) \; G^{\alpha, \beta}_{-\lambda}(-x) \; G^{\alpha, \beta}_{-\lambda}(y) \;  d\sigma_{\alpha, \beta}(\lambda) \\
& = & \int_{\mathbb{R}} e^{-\frac{t}{2}\left( \lambda^2+\rho^2\right)}  G^{\alpha, \beta}_{\lambda}(x) \; G^{\alpha, \beta}_{\lambda}(-y) \;  d\sigma_{\alpha, \beta}(\lambda).
\end{eqnarray*}
(b) Since $G^{\alpha, \beta}_\lambda(0)=1$, we obtain
\begin{eqnarray*}
p_t^{\alpha, \beta}(x, 0) =  \tau^{\alpha, \beta}_{-x} \mathcal{F}_t^{\alpha, \beta}(0)
& = & \int_{\mathbb{R}} e^{-\frac{t}{2}\left( \lambda^2+\rho^2\right)}  G^{\alpha, \beta}_{\lambda}(x) \; G^{\alpha, \beta}_{\lambda}(0) \;  d\sigma_{\alpha, \beta}(\lambda) \\
& = & \int_{\mathbb{R}} e^{-\frac{t}{2}\left( \lambda^2+\rho^2\right)}  G^{\alpha, \beta}_{\lambda}(x) \;  d\sigma_{\alpha, \beta}(\lambda) 
= \mathcal{F}_t^{\alpha, \beta}(x).
\end{eqnarray*}
(c) We have
\begin{eqnarray*}
p_t^{\alpha, \beta}(x, y) 
& = & \int_{\mathbb{R}} e^{-\frac{t}{2}\left( \lambda^2+\rho^2\right)}  G^{\alpha, \beta}_{\lambda}(x) \; G^{\alpha, \beta}_{\lambda}(-y) \;  d\sigma_{\alpha, \beta}(\lambda) \\
& = & \int_{\mathbb{R}} e^{-\frac{t}{2}\left( \lambda^2+\rho^2\right)}  G^{\alpha, \beta}_{\lambda}(-y) \;G^{\alpha, \beta}_{\lambda}(-(-x)) \;    d\sigma_{\alpha, \beta}(\lambda)
= p_t^{\alpha, \beta}(-y, -x).
\end{eqnarray*}
(d) From Proposition \ref{pro2}, we have $\mathcal{F}_t^{\alpha, \beta} \in \S^r_{\alpha, \beta} (\mathbb{R})$. Since the weighted Schwartz space $\S^r_{\alpha, \beta} (\mathbb{R})$ is invariant under these translation operators $\tau^{\alpha, \beta}_{x}$, using (\ref{eq1}), we obtain $p_t^{\alpha, \beta}(x, \cdot) \in \S^r_{\alpha, \beta} (\mathbb{R})$. Using the definition of $\H^{-1}_{\alpha, \beta}$, we get $\H^{-1}_{\alpha, \beta} \left( e^{-\frac{t}{2}\left( \lambda^2+\rho^2\right)} G^{\alpha, \beta}_\lambda(x) \right) (y)= p_t^{\alpha, \beta}(x, -y)$, and the second part follows immediately.\\
(e) From (d), we have 
\[ \H_{\alpha, \beta} \left( p_t^{\alpha, \beta}(x, - \; \cdot) \right) (\lambda) = e^{-\frac{t}{2}\left( \lambda^2+\rho^2\right)} G^{\alpha, \beta}_\lambda(x). \]
Therefore,
\[ \int_{\R} p_t^{\alpha, \beta}(x, -y) \; G^{\alpha, \beta}_\lambda(-y) \; A_{\alpha, \beta} (y) dy = e^{-\frac{t}{2}\left( \lambda^2+\rho^2\right)} G^{\alpha, \beta}_\lambda(x). \]
Using the change of variables, we get
\[\int_{\R} p_t^{\alpha, \beta}(x, y) \; G^{\alpha, \beta}_\lambda(y) \; A_{\alpha, \beta} (y) dy = e^{-\frac{t}{2}\left( \lambda^2+\rho^2\right)} G^{\alpha, \beta}_\lambda(x). \]
Since $G^{\alpha, \beta}_0(\cdot)=1$, by taking $\lambda=0$, we obtain
\[ \int_{\R} p_t^{\alpha, \beta}(x, y) A_{\alpha, \beta} (y) dy = e^{-\frac{t}{2} \rho^2}. \] 
(f) Finally, using differentiation under the integral sign in (a), interchanging the operator $T^2_{\alpha, \beta}$ with the integral sign, and using (\ref{eq02}), we get  $\mathrm{H} p_t^{\alpha, \beta}(x, y) = 0$.
\end{proof}

For $f \in \S^r_{\alpha, \beta} (\mathbb{R})$ and $t \geq 0$, we define the Jacobi--Cherednik heat semigroup $\left(P^{\alpha, \beta}_t, \;t \geq 0\right)$ by
\begin{equation}\label{eq5}
P^{\alpha, \beta}_t f(x)= 
\left\{\begin{array}{cl}
\displaystyle\int_{\R} p^{\alpha, \beta}_t(x, y) f(y) A_{\alpha, \beta} (y)dy  & \text { if } t>0, \\ 
f(x) & \text { if } t=0 .
\end{array}\right.
\end{equation}
From part (d) of Proposition \ref{pro3}, we have $p_t^{\alpha, \beta}(x, \cdot) \in L^1(\R, A_{\alpha, \beta})$ for each $t>0$. Since $f$ is bounded, $P^{\alpha, \beta}_t f$ is well defined and continuous on $\mathbb{R}$. Next, we show that $P^{\alpha, \beta}_t f(x)$ solves the Cauchy problem (\ref{eq07}).
\begin{theorem}\label{th1}
Let $f \in \S^r_{\alpha, \beta} (\mathbb{R})$. Then, the function $u(x, t)=P^{\alpha, \beta}_t f(x)$ on $\mathbb{R} \times (0, \infty)$, solves the Cauchy problem $(\ref{eq07})$. Also, $P^{\alpha, \beta}_t f$ satisfies the following properties:
\begin{itemize}
\item[(a)] $P^{\alpha, \beta}_t f \in \S^r_{\alpha, \beta} (\mathbb{R})$ for each $t>0$.

\item[(b)] $P^{\alpha, \beta}_{t+s} f =P^{\alpha, \beta}_t  P^{\alpha, \beta}_s f$ for every $s, t \geq 0$.

\item[(c)] $\| P^{\alpha, \beta}_t  f - f \|_{\infty} \rightarrow 0$ if $t \rightarrow 0$.
\end{itemize} 
\end{theorem}
\begin{proof}
Using Proposition \ref{pro3} (a) and Fubini's theorem, we obtain
\begin{eqnarray}\label{eq6}
u(x, t) & = & P^{\alpha, \beta}_t f(x) \nonumber \\
& = & \int_{\R} p^{\alpha, \beta}_t(x, y) f(y) A_{\alpha, \beta} (y)dy  \nonumber \\
& = & \int_{\R} \left( \int_{\mathbb{R}} e^{-\frac{t}{2}\left( \lambda^2+\rho^2\right)}  G^{\alpha, \beta}_\lambda(x)\;  G^{\alpha, \beta}_\lambda(-y) \; d\sigma_{\alpha, \beta}(\lambda) \right) f(y) A_{\alpha, \beta} (y)dy \nonumber \\
& = & \int_{\R} e^{-\frac{t}{2}\left( \lambda^2+\rho^2\right)} G^{\alpha, \beta}_\lambda(x) \left(\int_{\mathbb{R}} f(y) \; G^{\alpha, \beta}_\lambda(-y) A_{\alpha, \beta} (y)dy \right) d\sigma_{\alpha, \beta}(\lambda) \nonumber \\
& = & \int_{\R} e^{-\frac{t}{2}\left( \lambda^2+\rho^2\right)}  \H_{\alpha, \beta} (f)(\lambda) \; G^{\alpha, \beta}_\lambda(x) \; d\sigma_{\alpha, \beta}(\lambda) \nonumber \\
& = & \H^{-1}_{\alpha, \beta} \left( e^{-\frac{t}{2}\left( (\cdot)^2+\rho^2\right)}  \H_{\alpha, \beta} (f) \right) (x).
\end{eqnarray}
For every $t>0$, the function $\lambda \mapsto e^{-\frac{t}{2}\left( \lambda^2+\rho^2\right)} \H_{\alpha, \beta} (f) (\lambda)$ is in $\S\left(\Omega_{\varepsilon(r)}\right)$. Using the relation (\ref{eq2}), from (\ref{eq6}), we get (a) and
\begin{equation}\label{eq7}
\H_{\alpha, \beta} \left( u( \cdot , t) \right) (\lambda)= e^{-\frac{t}{2}\left( \lambda^2+\rho^2\right)}  \H_{\alpha, \beta} (f)(\lambda).
\end{equation}
Now, we can interchange the operator $T^2_{\alpha, \beta}$ and $\frac{\partial}{\partial t}$ respectively with the integral sign, in the following relation
\begin{equation}\label{eq8}
u(x , t) = \int_{\R} e^{-\frac{t}{2}\left( \lambda^2+\rho^2\right)}  \H_{\alpha, \beta} (f)(\lambda) \; G^{\alpha, \beta}_\lambda(x) \; d\sigma_{\alpha, \beta}(\lambda).
\end{equation}
Hence, using the relation (\ref{eq02}), we obtain that $u(x, t)$ solves the Cauchy problem $(\ref{eq07})$. For every $s, t \geq 0$, from the relation (\ref{eq7}), we get
\begin{eqnarray*}
&& \H_{\alpha, \beta} \left( P^{\alpha, \beta}_{t+s} f \right) (\lambda) = e^{-\frac{(t+s)}{2}\left( \lambda^2+\rho^2\right)}  \H_{\alpha, \beta} (f)(\lambda)
= e^{-\frac{t}{2}\left( \lambda^2+\rho^2\right)} e^{-\frac{s}{2}\left( \lambda^2+\rho^2\right)}  \H_{\alpha, \beta} (f)(\lambda) \\
&& =  e^{-\frac{t}{2}\left( \lambda^2+\rho^2\right)} \H_{\alpha, \beta} \left(P^{\alpha, \beta}_s f \right)(\lambda)
= \H_{\alpha, \beta} \left( P^{\alpha, \beta}_t P^{\alpha, \beta}_s f \right)(\lambda).
\end{eqnarray*}
Using the injectivity of $\H_{\alpha, \beta}$ on $\S^r_{\alpha, \beta} (\mathbb{R})$, we get (b). Since $\H_{\alpha, \beta} (f) \in L^1(\R, \sigma_{\alpha, \beta})$, using the relation (\ref{eq8}) and $|G^{\alpha, \beta}_\lambda(x)| \leq 1$, we obtain the following estimation
\[ \| P^{\alpha, \beta}_t  f - f \|_{\infty} \leq \int_{\R}   |\H_{\alpha, \beta} (f)(\lambda)| \; \left| e^{-\frac{t}{2}\left( \lambda^2+\rho^2\right)} -1 \right| d|\sigma_{\alpha, \beta}|(\lambda). \]
This integral tends to $0$ if $t \rightarrow 0$, and this completes the proof.
\end{proof} 

Next, we show that the linear operators $\left(P^{\alpha, \beta}_t, \;t \geq 0\right)$ on $\S^r_{\alpha, \beta} (\mathbb{R})$ extend to a positive contraction semigroup on a subspace of the Banach space $\left(C_0(\mathbb{R}), \|\cdot \|_{\infty}\right)$, where $C_0(\mathbb{R})$ is the space of continuous functions on $\R$ vanishing at infinity.

\begin{theorem}\label{th2}
Let $C_{\rho,0}(\mathbb{R})$ be  a subspace of $C_0(\mathbb{R})$ consisting of functions satisfying the condition $(\ref{eq9})$. Then, we have the following properties: 
\begin{itemize}
\item[(a)] The differential-difference operator $T^2_{\alpha, \beta}$ on $C_{\rho,0}(\mathbb{R})$ is closable, and its closure $\overline{T^2_{\alpha, \beta}}$   generates a positive, strongly continuous contraction semigroup $\left(\overline{P}^{\alpha, \beta}_t, \;t \geq 0\right)$ on $C_{\rho,0}(\mathbb{R})$.

\item[(b)] The heat kernel $p^{\alpha, \beta}_t$ is non-negative on $\mathbb{R} \times \mathbb{R}$.

\item[(c)] The action of $\overline{P}^{\alpha, \beta}_t$   on $C_{\rho,0}(\mathbb{R})$ is given by 
\[ \overline{P}^{\alpha, \beta}_t f(x) =\int_{\R} p^{\alpha, \beta}_t(x, y) f(y) A_{\alpha, \beta} (y)dy.  \]
\end{itemize}
\end{theorem}
\begin{proof}
(a) To prove the theorem and characterize generators of positive one-parameter contraction semigroups, we use a version of the Hille--Yosida--Phillip theorem (see \cite{are86, ros98}). To apply this theorem, we need to check the following two properties:
\begin{itemize}
\item[(i)] The operator $T^2_{\alpha, \beta}$ satisfies the following dispersivity condition: Let $f \in \S^r_{\alpha, \beta} (\mathbb{R})$ be a real-valued function, which satisfies the condition $(\ref{eq9})$ and attains an absolute maximum at $x_0 \in \R$. Then $T^2_{\alpha, \beta} f \left( x_0 \right) \leq 0$. 

\item[(ii)] For some $\lambda>0$, the range of $\lambda I - T^2_{\alpha, \beta}$ is dense in $C_{\rho, 0}(\mathbb{R})$.
\\
(i) follows from Proposition \ref{pro1}. Using (\ref{eq2}) and (\ref{eq06}), we prove (ii). Let $h \in  \S^r_{\alpha, \beta} (\mathbb{R})$, then $\H_{\alpha, \beta} (h) \in \S\left(\Omega_{\varepsilon(r)}\right)$. For every $\lambda>0$, the function $y \mapsto \frac{1}{\lambda+y^2}  \H_{\alpha, \beta} (h)(y)$ is in $\S\left(\Omega_{\varepsilon(r)}\right)$. Hence, from (\ref{eq2}), there exists $f \in \S^r_{\alpha, \beta} (\mathbb{R})$ such that $\H_{\alpha, \beta} (f)(y)=\frac{1}{\lambda+y^2} \H_{\alpha, \beta}(h)(y)$. Using (\ref{eq06}), we get $\H_{\alpha, \beta} (( \lambda I - T^2_{\alpha, \beta})f)(y)=\H_{\alpha, \beta} (h)(y)$ for every $y \in \mathbb{R}$. Thus $h=( \lambda I - T^2_{\alpha, \beta} )f$. Therefore, $(\lambda I - T^2_{\alpha, \beta} ) \S^r_{\alpha, \beta} (\mathbb{R})=\S^r_{\alpha, \beta} (\mathbb{R})$. Let $\mathcal{D}(\R)$ be the space of $C^{\infty}$ compactly supported functions on $\R$. Then $\mathcal{D}(\R) \subset C_c(\R)$. Now, the range of $\lambda I - T^2_{\alpha, \beta}$ is dense in $C_{\rho, 0}(\mathbb{R})$, which follows from the fact that $\mathcal{D}(\R) \subset \S^r_{\alpha, \beta} (\mathbb{R})$ and the density of $\mathcal{D}(\R)$ in $C_0(\mathbb{R})$.
\end{itemize}

(b) For every $f \in \S^r_{\alpha, \beta} (\mathbb{R})$ with $f \geq 0$, using the Hille--Yosida--Phillip theorem, we get
\[\int_{\R} p^{\alpha, \beta}_t(x, y) f(y) A_{\alpha, \beta} (y)dy = P^{\alpha, \beta}_t f(x) \geq 0,  \]
for all $(x, t) \in \R \times (0, \infty)$. Since for each fixed $x \in \mathbb{R}$ and $t>0$ the function $y \mapsto p^{\alpha, \beta}_t(x, y)$ is continuous on $\mathbb{R}$, the heat kernel $p^{\alpha, \beta}_t(x, y) \geq 0$ on $\mathbb{R} \times \mathbb{R}$ for each $t>0$.

(c) For $f \in C_{\rho, 0}(\mathbb{R})$, using Proposition \ref{pro3} (e) and the positivity of $p^{\alpha, \beta}_t$,  we obtain 
\[ \left| \int_{\R} p^{\alpha, \beta}_t(x, y) f(y) A_{\alpha, \beta} (y)dy \right| \leq \| f \|_{\infty} . \]
This completes the proof.
\end{proof}

\begin{corollary}
For each $t > 0$, the Jacobi--Cherednik heat kernel $p^{\alpha, \beta}_t(x, y)$ is strictly positive on $\mathbb{R} \times \mathbb{R}$.
\end{corollary}
\begin{proof}
From Theorem \ref{th2} (b), the heat kernel $p^{\alpha, \beta}_t$ is non-negative on $\R \times \R$. For $\lambda \in \R$, the function $G^{\alpha, \beta}_\lambda$ is real and strictly positive. Now, we have
\[ p_t^{\alpha, \beta}(x, y)= \int_{\mathbb{R}} e^{-\frac{t}{2}\left( \lambda^2+\rho^2\right)}  G^{\alpha, \beta}_\lambda(x)\;  G^{\alpha, \beta}_\lambda(-y) \; d\sigma_{\alpha, \beta}(\lambda). \]
Also, for $\lambda = 0$, the integrand  
\[ e^{-\frac{t}{2}\left( \lambda^2+\rho^2\right)}  G^{\alpha, \beta}_\lambda(x)\;  G^{\alpha, \beta}_\lambda(-y)= e^{-\frac{t}{2} \rho^2}  G^{\alpha, \beta}_0(x)\;  G^{\alpha, \beta}_0(-y) = e^{-\frac{t}{2}  \rho^2 } >0,  \]
for all $t>0$. Thus, the integrand on the right side is continuous, non-negative and not identically zero. Therefore the integral itself must be strictly positive. 
\end{proof}  

\begin{theorem}\label{th3}
The Jacobi--Cherednik heat semigroup $\left(P^{\alpha, \beta}_t, \;t \geq 0\right)$ defines a strongly continuous, positivity-preserving contraction semigroup on $L^2(\R, A_{\alpha, \beta})$. The generator of this semigroup is the self-adjoint closure $\overline{T^2_{\alpha, \beta}}$ of $T^2_{\alpha, \beta}$ on $L^2(\R, A_{\alpha, \beta})$. Also, the action of $\overline{P}^{\alpha, \beta}_t$ on $L^2(\R, A_{\alpha, \beta})$ is given by
\begin{equation}\label{eq10}
\overline{P}^{\alpha, \beta}_t f(x) = \int_{\mathbb{R}} e^{-\frac{t}{2}\left( \lambda^2+\rho^2\right)} \; \H_{\alpha, \beta} (f) (\lambda)\;  G^{\alpha, \beta}_\lambda(x)\; d\sigma_{\alpha, \beta}(\lambda).
\end{equation}
\end{theorem}
\begin{proof}
For $f \in \S^r_{\alpha, \beta} (\mathbb{R})$, using (\ref{eq6}), we have $\overline{P}^{\alpha, \beta}_t f = P^{\alpha, \beta}_t f$. Let $f \in L^2(\R, A_{\alpha, \beta})$. Then, there exists a sequence $\{ f_n \}_{n \in \N}$ in $\S^r_{\alpha, \beta} (\mathbb{R})$ such that $f_n \rightarrow f$ in $L^2(\R, A_{\alpha, \beta})$. Then, using the Plancherel theorem, we obtain that $\left\{ \overline{P}^{\alpha, \beta}_t f_n \right\}_{n \in \N}$ is a Cauchy sequence in $L^2(\R, A_{\alpha, \beta})$ and hence it converges in $L^2(\R, A_{\alpha, \beta})$. Then, for a subsequence $\{n_k\}$, the limit
\[ \lim _{n_k \rightarrow \infty} \int_{\mathbb{R}} e^{-\frac{t}{2}\left( \lambda^2+\rho^2\right)} \; \H_{\alpha, \beta} (f_{n_k}) (\lambda)\;  G^{\alpha, \beta}_\lambda(x)\; d\sigma_{\alpha, \beta}(\lambda) = l(x) \]
exists for almost all $x \in \mathbb{R}$. Since $\H_{\alpha, \beta} (f_{n_k}) \rightarrow \H_{\alpha, \beta} (f)$ in $L^2(\R, \sigma_{\alpha, \beta})$, using the dominated convergence theorem, we get 
\[ l(x) = \int_{\mathbb{R}} e^{-\frac{t}{2}\left( \lambda^2+\rho^2\right)} \; \H_{\alpha, \beta} (f) (\lambda)\;  G^{\alpha, \beta}_\lambda(x)\; d\sigma_{\alpha, \beta}(\lambda). \] 
Therefore, $\overline{P}^{\alpha, \beta}_t f$ defined by (\ref{eq10}), is in $L^2(\R, A_{\alpha, \beta})$. Then, using the density of $\S^r_{\alpha, \beta} (\mathbb{R})$ in $ L^2(\R, A_{\alpha, \beta})$ and the semigroup property of $\left(P^{\alpha, \beta}_t, \;t \geq 0\right)$, we obtain that   $\left(\overline{P}^{\alpha, \beta}_t, \;t \geq 0\right)$ is a semigroup of operators on $L^2(\R, A_{\alpha, \beta})$. Also, we have 
\[ \left\| \H^{-1}_{\alpha, \beta} \left( e^{-\frac{t}{2}\left( \lambda^2+\rho^2\right)} \; \H_{\alpha, \beta} (f)\right) - f  \right\|_{L^2(\R, A_{\alpha, \beta})} = \left\| e^{-\frac{t}{2}\left( \lambda^2+\rho^2\right)} \; \H_{\alpha, \beta} (f) - \H_{\alpha, \beta} (f)  \right\|_{L^2(\R, \sigma_{\alpha, \beta})} \rightarrow 0, \]
as $t \rightarrow 0$. Hence, $\left(\overline{P}^{\alpha, \beta}_t, \;t \geq 0\right)$ is a strongly continuous semigroup on $L^2(\R, A_{\alpha, \beta})$.
\end{proof}

\section{Image of $L^2(\R, A_{\alpha, \beta})$ under the Jacobi--Cherednik heat semigroup}\label{sec5}

In this section, we characterize the image of the space $L^2(\R, A_{\alpha, \beta})$ under the Jacobi--Cherednik heat semigroup $\left(P^{\alpha, \beta}_t, \;t \geq 0\right)$ given by (\ref{eq5}). From (\ref{eq6}), we can write $P^{\alpha, \beta}_t f$ as
\begin{equation}\label{eq15}
P^{\alpha, \beta}_t f(x) = \int_{\R} e^{-\frac{t}{2}\left( \lambda^2+\rho^2\right)}  \H_{\alpha, \beta} (f)(\lambda) \; G^{\alpha, \beta}_\lambda(x) \; d\sigma_{\alpha, \beta}(\lambda).
\end{equation}
For $f \in L^2(\R, A_{\alpha, \beta})$, using $|G^{\alpha, \beta}_\lambda(x) | \leq 1$, H\"older's inequality and Plancherel's formula (\ref{eq03}), we obtain
\begin{eqnarray*}
\left| P^{\alpha, \beta}_t f(x) \right| 
& \leq & \int_{\R} e^{-\frac{t}{2}\left( \lambda^2+\rho^2\right)}  \left| \H_{\alpha, \beta} (f)(\lambda) \right|  d\sigma_{\alpha, \beta}(\lambda) \\
& \leq &  \left( \int_{\R} e^{-t \left( \lambda^2+\rho^2\right)} d\sigma_{\alpha, \beta}(\lambda) \right)^{1/2}
\left( \int_{\R} \left| \H_{\alpha, \beta} (f)(\lambda) \right|^2  d\sigma_{\alpha, \beta}(\lambda) \right)^{1/2} \\
& \leq & c_{\alpha, \beta} \; \Vert f \Vert_{ L^2(\R, A_{\alpha, \beta})},
\end{eqnarray*}
where $c_{\alpha, \beta}$ is a positive constant. Now, using Morera's and Dominated convergence theorems, we see that $P^{\alpha, \beta}_t f$ can be extended as an analytic function on $\C$. Therefore, $P^{\alpha, \beta}_t \left( L^2(\R, A_{\alpha, \beta}) \right)$ is a subspace of space of all analytic functions on $\C$. Further, it is obvious that $P^{\alpha, \beta}_t : L^2(\R, A_{\alpha, \beta}) \to P^{\alpha, \beta}_t \left( L^2(\R, A_{\alpha, \beta}) \right) $ is linear and bijective. Hence, $P^{\alpha, \beta}_t \left( L^2(\R, A_{\alpha, \beta})\right)$ can be made into a Hilbert space simply by transferring the Hilbert space structure of $ L^2(\R, A_{\alpha, \beta})$ to $P^{\alpha, \beta}_t \left( L^2(\R, A_{\alpha, \beta}) \right)$ so that the Jacobi--Cherednik heat semigroup $ P^{\alpha, \beta}_t $ is an isometric isomorphism from $L^2(\R, A_{\alpha, \beta})$ onto $P^{\alpha, \beta}_t \left( L^2(\R, A_{\alpha, \beta}) \right)$. We give the inner product structure in $P^{\alpha, \beta}_t \left( L^2(\R, A_{\alpha, \beta}) \right)$ as follows:
\[ \left\langle  P^{\alpha, \beta}_t f, P^{\alpha, \beta}_t g \right\rangle_{P^{\alpha, \beta}_t \left( L^2(\R, A_{\alpha, \beta}) \right)}:= \langle f , g \rangle_{L^2(\R, A_{\alpha, \beta})} \quad \text{for} \; f,g \in L^2(\R, A_{\alpha, \beta}). \]
Here, we prove that the space $P^{\alpha, \beta}_t \left( L^2(\R, A_{\alpha, \beta}) \right)$ is a reproducing kernel Hilbert space. 

The convolution product associated with the Opdam--Cherednik transform is defined for two suitable functions $f$ and $g$  by \cite{ank12}
\[ (f *_{\alpha, \beta} g) (x)=\int_\R  f(y) \;\tau_x^{\alpha, \beta} g(-y) \; A_{\alpha, \beta}(y) \; dy.  \]
Also, we have
\begin{equation}\label{eq16}
\H_{\alpha, \beta} (f *_{\alpha, \beta} g)= \H_{\alpha, \beta} (f) \; \H_{\alpha, \beta} (g).
\end{equation}
From the relation (\ref{eq08}), we get
\begin{equation}\label{eq17}
\H_{\alpha, \beta} \left(\mathcal{F}_t^{\alpha, \beta} \right) (\lambda) =  e^{-\frac{t}{2}\left( \lambda^2 + \rho^2 \right)}, \quad \lambda \in \R, \; t>0.
\end{equation}
Now, using (\ref{eq17}) in the relation (\ref{eq15}), we obtain
\begin{eqnarray}\label{eq18}
P^{\alpha, \beta}_t f(x) 
& = &  \int_{\R} \H_{\alpha, \beta} \left(\mathcal{F}_t^{\alpha, \beta} \right) (\lambda) \;   \H_{\alpha, \beta} (f)(\lambda) \; G^{\alpha, \beta}_\lambda(x) \; d\sigma_{\alpha, \beta}(\lambda) \nonumber \\ 
& = & \int_{\R} \H_{\alpha, \beta} \left( f *_{\alpha, \beta} \mathcal{F}_t^{\alpha, \beta} \right) (\lambda) \; G^{\alpha, \beta}_\lambda(x) \; d\sigma_{\alpha, \beta}(\lambda) \nonumber \\ 
& = & \H_{\alpha, \beta}^{-1} \left( \H_{\alpha, \beta} \left( f *_{\alpha, \beta} \mathcal{F}_t^{\alpha, \beta} \right) \right) (x) \nonumber \\ 
& = & \left( f *_{\alpha, \beta} \mathcal{F}_t^{\alpha, \beta} \right) (x).
\end{eqnarray}
For $f \in  L^2(\R, A_{\alpha, \beta})$, $\left( f *_{\alpha, \beta} \mathcal{F}_t^{\alpha, \beta} \right) (x)$ has holomorphic extension on $\C$, where 
\[ \left( f *_{\alpha, \beta} \mathcal{F}_t^{\alpha, \beta} \right) (x) =\int_\R  f(y) \;\tau_x^{\alpha, \beta} \mathcal{F}_t^{\alpha, \beta}(-y) \; A_{\alpha, \beta}(y) \; dy. \]
Let $\mathcal{O}(\C) $ be the space of all analytic functions on $\C$. The heat semigroup $P^{\alpha, \beta}_t : L^2(\R, A_{\alpha, \beta}) \to \mathcal{O}(\C) $ such that $f \mapsto f *_{\alpha, \beta} \mathcal{F}_t^{\alpha, \beta}$, is one to one. 
\begin{theorem}\label{th5}
The space $P^{\alpha, \beta}_t \left( L^2(\R, A_{\alpha, \beta}) \right)$ is a reproducing kernel Hilbert space with kernel
\[ K^{\alpha, \beta}_t (z,u) = p^{\alpha, \beta}_{2t} (-z,u), \quad z,u \in \C. \]
\end{theorem}
\begin{proof}
We give the inner product structure in $P^{\alpha, \beta}_t \left( L^2(\R, A_{\alpha, \beta}) \right)$ as follows: 
\[ \left\langle  F, G \right\rangle_{P^{\alpha, \beta}_t \left( L^2(\R, A_{\alpha, \beta}) \right)}:= \langle f , g \rangle_{L^2(\R, A_{\alpha, \beta})},  \]
where $f$ and $g$ are preimages of $F$ and $G$ respectively under the heat semigroup $P^{\alpha, \beta}_t$. Since $P^{\alpha, \beta}_t \left( L^2(\R, A_{\alpha, \beta}) \right)$ is a Hilbert space of analytic functions, the point evaluations are continuous. Therefore $P^{\alpha, \beta}_t \left( L^2(\R, A_{\alpha, \beta}) \right)$ is a reproducing kernel Hilbert space. For $F \in P^{\alpha, \beta}_t \left( L^2(\R, A_{\alpha, \beta}) \right)$, there exists $f \in L^2(\R, A_{\alpha, \beta})$ such that for $u \in \C$, we have
\begin{eqnarray*}
F(u) 
& = & \left( f *_{\alpha, \beta} \mathcal{F}_t^{\alpha, \beta} \right) (u)  \\
& = & \int_\R  f(y) \;\tau_u^{\alpha, \beta} \mathcal{F}_t^{\alpha, \beta}(-y) \; A_{\alpha, \beta}(y) \; dy \\
& = & \int_\R  f(y) \; p^{\alpha, \beta}_{t} (-u, -y) \; A_{\alpha, \beta}(y) \; dy \\
& = & \left\langle f, \; \overline{p^{\alpha, \beta}_{t} (-u, - \; \cdot)} \right\rangle_{L^2(\R, A_{\alpha, \beta})}.
\end{eqnarray*}
Since $F(u) = \left\langle F, K^{\alpha, \beta}_t (\cdot,u) \right\rangle_{P^{\alpha, \beta}_t \left( L^2(\R, A_{\alpha, \beta}) \right)}$, we get
\begin{eqnarray*}
K^{\alpha, \beta}_t (z,u)
& = & P^{\alpha, \beta}_t \left( \overline{p^{\alpha, \beta}_{t} (-u, - \; \cdot)} \right) (z) \\
& = & \left( \overline{p^{\alpha, \beta}_{t} (-u, - \; \cdot)} *_{\alpha, \beta} \mathcal{F}_t^{\alpha, \beta} \right) (z) \\
& = & \int_\R  \overline{p^{\alpha, \beta}_{t} (-u, - y)} \;\tau_z^{\alpha, \beta} \mathcal{F}_t^{\alpha, \beta}(-y) \; A_{\alpha, \beta}(y) \; dy \\
& = & \int_\R  \overline{p^{\alpha, \beta}_{t} (-u, y)} \;\tau_z^{\alpha, \beta} \mathcal{F}_t^{\alpha, \beta}(y) \; A_{\alpha, \beta}(y) \; dy \\
& = & \int_\R  p^{\alpha, \beta}_{t} (-z, y) \; \overline{p^{\alpha, \beta}_{t} (-u, y)} \; A_{\alpha, \beta}(y) \; dy \\
& = & \left\langle p^{\alpha, \beta}_{t} (-z, \cdot) , \;  p^{\alpha, \beta}_{t} (-u, \cdot) \right\rangle_{L^2(\R, A_{\alpha, \beta})} \\
& = & \left\langle \H_{\alpha, \beta}^{-1} \left( e^{-\frac{t}{2}\left( \lambda^2+\rho^2\right)}  \; G^{\alpha, \beta}_\lambda(-z)  \right) (- \; \cdot), \; \H_{\alpha, \beta}^{-1} \left( e^{-\frac{t}{2}\left( \lambda^2+\rho^2\right)}  \; G^{\alpha, \beta}_\lambda(-u)  \right) (- \; \cdot)  \right\rangle_{L^2(\R, A_{\alpha, \beta})}  \\
& = & \left\langle  e^{-\frac{t}{2}\left( \lambda^2+\rho^2\right)}  \; G^{\alpha, \beta}_\lambda(-z), \;  e^{-\frac{t}{2}\left( \lambda^2+\rho^2\right)}  \; G^{\alpha, \beta}_\lambda(-u) \right\rangle_{L^2(\R, \sigma_{\alpha, \beta})}\\
& = & \int_{\R} e^{-\frac{t}{2}\left( \lambda^2+\rho^2\right)}  \; G^{\alpha, \beta}_\lambda(-z) \;  e^{-\frac{t}{2}\left( \lambda^2+\rho^2\right)} \; \overline{G^{\alpha, \beta}_\lambda(-u)} \; d\sigma_{\alpha, \beta}(\lambda)\\
& = & \int_{\R} e^{-t \left( \lambda^2+\rho^2\right)}  \; G^{\alpha, \beta}_\lambda(-z) \;  G^{\alpha, \beta}_\lambda(-u) \; d\sigma_{\alpha, \beta}(\lambda)\\
& = & p^{\alpha, \beta}_{2t} (-z,u).
\end{eqnarray*}
This completes the proof.
\end{proof} 
%%%%%%%%%%%%%%%%%%%%%%%%%%%%%%%%%%%%%%%%

\section{Applications}\label{sec4}

\subsection{The Jacobi--Cherednik--Markov processes}
In this subsection, we give a direct application of the Jacobi--Cherednik heat semigroup in terms of Markov theory and introduce a new family of one-dimensional Markov processes. We define the transition kernels $\mathcal{K}^{\alpha, \beta}_t$ associate to the heat semigroup $\left(P^{\alpha, \beta}_t, \;t \geq 0\right)$ by 
\begin{equation}\label{eq11}
\mathcal{K}^{\alpha, \beta}_t (x, dy)= p^{\alpha, \beta}_t(x, y) A_{\alpha, \beta} (y) dy, \quad t>0,
\end{equation}
which are probability measures on the Borel $\sigma$-field $\mathcal{B}_{\mathbb{R}}$, for all $x \in \mathbb{R}$ and $t>0$. If a Markov process $X = \left(X_t, \;t \geq 0\right)$ defined on some filtered probability space $\left(\Omega, \mathbb{F}, \left( \mathbb{F}_t, \;t \geq 0 \right), \mathbb{P} \right)$ has transition kernels given by (\ref{eq11}), then we call it a Jacobi--Cherednik process of index $(\alpha, \beta)$. Also, the properties of the Jacobi--Cherednik processes follow from the structure of the semigroup $\left(P^{\alpha, \beta}_t, \;t \geq 0\right)$. Moreover, the absolute value  $|X| = \left(|X_t|, \;t \geq 0\right)$ of a Jacobi--Cherednik process is a diffusion process on $\R_+$. In particular, we have the following result. 
\begin{proposition}
Let $ \left(X_t, \;t \geq 0\right)$ be a Jacobi--Cherednik process on $\mathbb{R}$ with transition kernels  given by (\ref{eq11}). Then the absolute value $ \left(|X_t|, \;t \geq 0\right)$ is a diffusion process on $\mathbb{R}_{+}$ with infinitesimal generator the operator 
\begin{equation}\label{eq12}
\frac{d^2}{dx^2} + \frac{A'_{\alpha, \beta} (x)}{A_{\alpha, \beta} (x)} \frac{d}{dx} + \rho^2 I.
\end{equation}
\end{proposition}
\begin{proof}
The proof follows from the fact that the restriction of $T^2_{\alpha, \beta}$ given by (\ref{eq05}) on even functions  is the operator of the form (\ref{eq12}). 
\end{proof}

If $\rho=0$, then the process $ \left(|X_t|, \;t \geq 0\right)$ with generator $\frac{d^2}{dx^2} + \frac{A'_{\alpha, \beta} (x)}{A_{\alpha, \beta} (x)} \frac{d}{dx}$ is called the Jacobi diffusion process of index $(\alpha, \beta)$ on $\mathbb{R}_{+}$ (see \cite{gru97, gru98}). Also, the Jacobi process is the radial part of the Brownian motion on a Riemannian symmetric space for some special values of $(\alpha, \beta)$ (see \cite{gru98}). In particular, for $\alpha =\frac{1}{2}$ and $\beta=-\frac{1}{2}$ the Jacobi process is the hyperbolic-Bessel process (see \cite{gru97, yor99}). Moreover, if the Jacobi--Cherednik process has a jump at time $s$, then $X_s=-X_{s^-}$ and this implies that $\left(|X_t|, \;t \geq 0\right)$ is a continuous process. The Dunkl processes, which are the Markov processes associated with the square of the Dunkl operator $D_\alpha$ were studied in \cite{yor06}. Furthermore, if $\left(P^{\alpha, \beta}_t, \;t > 0\right)$ is a Feller-semigroup, then the Jacobi--Cherednik processes have a version with c\'adl\`ag trajectories and satisfy the strong Markov property (see \cite{yor99}).
 
\subsection{The modified Poisson equation}

Here, we use the results on the heat semigroup to solve the modified Poisson equation $\frac{1}{2}\left(T^2_{\alpha, \beta} - \rho^2 \right) u=-f$. 
 
\begin{theorem}\label{th4}
Let $f \in L^1(\R, A_{\alpha, \beta})$ such that $\H_{\alpha, \beta} (f) \in L^1(\R, \sigma_{\alpha, \beta})$. Then the function 
\begin{equation}\label{eq13}
\mathcal{G} f(x)= \int_0^{\infty} \int_{\R} p^{\alpha, \beta}_t(x, y) f(y) A_{\alpha, \beta}(y) \;dy dt
\end{equation}
is in $C^2(\R)$, bounded and satisfies the modified Poisson equation $\frac{1}{2}\left(T^2_{\alpha, \beta} - \rho^2 \right) \mathcal{G} f=-f$. 
\end{theorem}
\begin{proof}
First, we prove that $\mathcal{G} f$ is well-defined and bounded. For $f \geq 0$, using Proposition \ref{pro3} (a), we get
\[ \mathcal{G} f(x)= \int_0^{\infty} \int_{\R} \int_{\mathbb{R}}  e^{-\frac{t}{2}\left( \lambda^2+\rho^2\right)}  G^{\alpha, \beta}_\lambda(x)\;  G^{\alpha, \beta}_\lambda(-y)  f(y) A_{\alpha, \beta}(y) \; d\sigma_{\alpha, \beta}(\lambda) dy dt. \]
Since $f \in L^1(\R, A_{\alpha, \beta})$, using $|G^{\alpha, \beta}_\lambda(x) | \leq 1$, we obtain
\begin{eqnarray*}
&& \int_{\R} \int_{\mathbb{R}}  e^{-\frac{t}{2}\left( \lambda^2+\rho^2\right)} \left| G^{\alpha, \beta}_\lambda(x)\;  G^{\alpha, \beta}_\lambda(-y)  f(y) A_{\alpha, \beta}(y) \right| \; d\sigma_{\alpha, \beta}(\lambda) dy \\
&& \leq \int_{\R} \int_{\mathbb{R}}  e^{-\frac{t}{2}\left( \lambda^2+\rho^2\right)} f(y) A_{\alpha, \beta}(y) \; d\sigma_{\alpha, \beta}(\lambda) dy < \infty.  
\end{eqnarray*}
Now, using the definition of the Opdam--Cherednik transform and Fubini's theorem, we have
\begin{equation}\label{eq14}
\mathcal{G} f(x)= \int_0^{\infty} \int_{\R}  e^{-\frac{t}{2}\left( \lambda^2+\rho^2\right)}  G^{\alpha, \beta}_\lambda(x)\; \H_{\alpha, \beta} (f) (\lambda) \; d\sigma_{\alpha, \beta}(\lambda) dt.
\end{equation}
Hence,
\[ |\mathcal{G} f(x)| \leq \int_0^{\infty} \int_{\R}  e^{-\frac{t}{2}\left( \lambda^2+\rho^2\right)}  \left| \H_{\alpha, \beta} (f) (\lambda) \right| d\sigma_{\alpha, \beta}(\lambda) dt
= 2  \int_{\R} \frac{\left| \H_{\alpha, \beta} (f) (\lambda) \right|}{\lambda^2+\rho^2} \; d\sigma_{\alpha, \beta}(\lambda) <\infty. \]
Therefore, $\mathcal{G} f$ is well-defined and bounded on $\R$. Here, we call the operator $\mathcal{G}$ the Green operator. Now, applying Fubini's theorem in (\ref{eq14}), we obtain
\[ \mathcal{G} f(x)= 2  \int_{\R} \frac{\H_{\alpha, \beta} (f) (\lambda)}{\lambda^2+\rho^2} \; G^{\alpha, \beta}_\lambda(x) \; d\sigma_{\alpha, \beta}(\lambda). \]
For $x \in \R$, we have
\[ T^2_{\alpha, \beta} \mathcal{G} f(x)= 2  \int_{\R} \frac{\H_{\alpha, \beta} (f) (\lambda)}{\lambda^2+\rho^2} \; T^2_{\alpha, \beta}G^{\alpha, \beta}_\lambda(x) \; d\sigma_{\alpha, \beta}(\lambda). \]
Using the relation (\ref{eq02}) and the definition of the inverse Opdam--Cherednik transform, we get
\begin{eqnarray*}
T^2_{\alpha, \beta} \mathcal{G} f(x)
& = & - 2  \int_{\R} \frac{\H_{\alpha, \beta} (f) (\lambda)}{\lambda^2+\rho^2} \;  \lambda^2  G^{\alpha, \beta}_\lambda(x) \; d\sigma_{\alpha, \beta}(\lambda) \\
& = & - 2  \int_{\R} \H_{\alpha, \beta} (f) (\lambda) \;  \frac{\lambda^2 + \rho^2 -\rho^2}{\lambda^2+\rho^2} \;    G^{\alpha, \beta}_\lambda(x) \; d\sigma_{\alpha, \beta}(\lambda) \\
& = & - 2  \int_{\R} \H_{\alpha, \beta} (f) (\lambda) \; G^{\alpha, \beta}_\lambda(x) \; d\sigma_{\alpha, \beta}(\lambda) + 2 \rho^2 \int_{\R}  \; \frac{\H_{\alpha, \beta} (f) (\lambda)}{\lambda^2+\rho^2} \; G^{\alpha, \beta}_\lambda(x) \; d\sigma_{\alpha, \beta}(\lambda) \\
& = & - 2 f(x) + \rho^2 \mathcal{G} f(x).
\end{eqnarray*}
Hence, the function $\mathcal{G} f$ satisfies the modified Poisson equation $\frac{1}{2} \left( T^2_{\alpha, \beta} - \rho^2 \right) \mathcal{G} f = - f$.  
\end{proof}

\section*{Acknowledgments}
The first author is grateful to the Science and Engineering Research Board (SERB), Government of India for providing the National Post-Doctoral Fellowship [File No. PDF/2021/000192]. The first author is also partially supported by the XJTLU Research Development Fund (RDF-23-01-027). The authors wish to thank the referees for their helpful comments and suggestions that helped to improve the quality of the paper.

\section*{Conflict of interest}
The authors declare that there is no potential conflict of interest regarding the publication of this article.

\section*{Data Availability}
Data sharing is not applicable to this article as no new data were created or analyzed in this study.

\section*{ORCID}
{\it Anirudha Poria} https://orcid.org/0000-0002-0224-3642


\begin{thebibliography}{10}

\bibitem{ank12}
J.-P. Anker, F. Ayadi and M. Sifi, {\it Opdam's hypergeometric functions: product formula and convolution structure in dimension 1}, Adv Pure Appl Math. 3(1):11--44 (2012).

\bibitem{are86}
W. Arendt, {\it Characterization of positive semigroups on Banach lattices}, In: One-parameter semigroups of positive operators (R. Nagel (ed.)), Lecture Notes in Math. 1184, Springer-Verlag, Berlin, 247--291 (1986).

\bibitem{bar61}
V. Bargmann, {\it On a Hilbert space of analytic functions and an associated integral transform, part I}, Commun Pure Appl Math. 14:187--214 (1961).

\bibitem{chou06}
F. Chouchene, L. Gallardo and M. Mili, {\it The Heat Semigroup for the Jacobi–Dunkl Operator and the Related Markov Processes}, Potential Anal. 25(2):103--119 (2006).

%\bibitem{cho03}
%F. Chouchane, M. Mili and K. Trim{\`e}che, {\it Positivity of the intertwining operator and harmonic analysis associated with the Jacobi--Dunkl operator on $\mathbb{R}$}, Anal Appl. 1(4):387--412 (2003).

\bibitem{dun91}
C.F. Dunkl, {\it Integral kernels with reflection group invariance}, Canad J Math. 43(6):1213--1227 (1991).

%\bibitem{fit89} 
%A. Fitouhi, {\it Heat polynomials for a singular differential operator on $(0, \infty)$}, J Constr Approx. 5(2):241--270 (1989). 

\bibitem{yor06}
L. Gallardo and M. Yor, {\it Some remarkable properties of the Dunkl martingales}, In memoriam P.A. Meyer: S\'eminaire de Probabilit\'es XXXIX, Lecture Notes in Math. 1874, Springer, Berlin, 337--356 (2006).

%\bibitem{gro01} 
%K. Gr\"ochenig, {\it Foundations of Time-Frequency Analysis}, Birkh\"auser, Boston (2001).

\bibitem{gru97}
J.C. Gruet, {\it Windings of hyperbolic Brownian motion}, In: Exponential functionals and principal values related to Brownian motion, Bibl Rev Math Ibroamericana, Rev Math Ibroamericana, Madrid, 35--72 (1997).

\bibitem{gru98}
J.C. Gruet, {\it Jacobi radial stable processes}, Ann Math Blaise Pascal 5(2):39--48 (1998).

\bibitem{hal04}
B.C. Hall and W. Lewkeeratiyutkul, {\it Holomorphic Sobolev spaces and the generalized Segal--Bargmann transform}, J Funct Anal. 217(1):192--220 (2004).

\bibitem{hec91}
G.J. Heckman, {\it An elementary approach to the hypergeometric shift operators of Opdam}, Invent Math. 103:341--350 (1991).

\bibitem{hik96}
K. Hikami, {\it Dunkl operators formalism for quantum many-body problems associated with classical root systems}, Phys Soc Japan. 65:394--401 (1996).

\bibitem{mej14}
H. Mejjaoli, {\it  Spectral theorems associated with the Jacobi--Cherednik operator}, Bull Sci Math. 138(3):416--439 (2014).

\bibitem{mon221}
S.S. Mondal and A. Poria, {\it Qualitative uncertainty principles for the windowed Opdam--Cherednik transform on weighted modulation spaces}, Math Meth Appl Sci. 45(16):10424--10439 (2022).

\bibitem{mon222}
S.S. Mondal and A. Poria, {\it Hausdorff operators associated with the Opdam--Cherednik transform in Lebesgue spaces}, J Pseudo-Differ Oper Appl. 13(3): Article No. 31, 20 pp. (2022).

\bibitem{mon223}
S.S. Mondal and A. Poria, {\it Weighted norm inequalities for the Opdam--Cherednik transform}, Internat J Math. 33(9): 2250066, 21 pp. (2022).

\bibitem{mon23}
S.S. Mondal and A. Poria, {\it Uncertainty principles for the windowed Opdam--Cherednik transform}, Math Meth Appl Sci. 46(18):18861--18877 (2023).

\bibitem{opd95}
E.M. Opdam, {\it Harmonic analysis for certain representations of graded Hecke algebras}, Acta Math. 175(1):75--121 (1995).

\bibitem{opd00}
E.M. Opdam, {\it Lecture notes on Dunkl operators for real and complex reflection groups}, In: Mem Math Soc Japan 8, Tokyo (2000).

\bibitem{por21}
A. Poria, {\it Uncertainty principles for the Opdam--Cherednik transform on modulation spaces}, Integral Trans Spec Funct. 32(3):191--206 (2021).

\bibitem{por23}
A. Poria, {\it Localization operators associated with the windowed Opdam--Cherednik transform on modulation spaces}, Complex Var Elliptic Equ. 68(8):1361--1384 (2023).

\bibitem{rad09}
R. Radha and S. Thangavelu, {\it Holomorphic Sobolev spaces, Hermite and special Hermite semigroups and a Paley--Wiener theorem for the windowed Fourier transform}, J Math Anal Appl. 354(2):564--574 (2009).

\bibitem{yor99}
D. Revuz and M. Yor, {\it Continuous martingales and Brownian motion}, Grundlehren der mathematischen Wissenschaften 293, Third edition,  Springer-Verlag, Berlin (1999).

\bibitem{ros98} 
M. R\"osler, {\it Generalized Hermite Polynomials and the Heat Equation for Dunkl Operators}, Commun Math Phys. 192:519--541 (1998).

\bibitem{sch08}
B. Schapira, {\it Contributions to the hypergeometric function theory of Heckman and Opdam: sharp estimates, Schwartz space, heat kernel}, Geom Funct Anal. 18:222--250 (2008).

\bibitem{tha07}
S. Thangavelu, {\it Holomorphic Sobolev spaces associated to compact symmetric spaces}, J Funct Anal. 251(2):438--462 (2007).

\bibitem{die00}
J.F. Van Diejen and L. Vinet, {\it Calogero--Moser--Sutherland Models},  CRM Series in Mathematical Physics, New York: Springer; 2000.

\end{thebibliography}
\end{document}